\documentclass[preprint,12pt]{elsarticle}

\usepackage{amssymb}
\usepackage{amsmath}
\PassOptionsToPackage{table}{xcolor}

\usepackage{hyperref}
\usepackage{natbib}
\usepackage{epstopdf}
\usepackage{url}
\usepackage[titletoc]{appendix}
\usepackage{eurosym}
\usepackage{listings}
\usepackage{setspace}
\usepackage{float}
\usepackage{booktabs}
\usepackage{pdfpages}
\usepackage{array}
\usepackage{placeins}
\usepackage{amsthm} 
\usepackage{graphicx}
\usepackage{mathrsfs}

\theoremstyle{definition}

\hypersetup{
    colorlinks,
    citecolor=black,
    filecolor=black,
    linkcolor=black,
    urlcolor=black
}

\journal{arxiv.org}

\begin{document}

\newtheorem{theorem}{Theorem}
\newtheorem{corollary}{Corollary}
\newtheorem{lemma}{Lemma}
\theoremstyle{definition}
\newtheorem{definition}{Definition}

\begin{frontmatter}



\title{Vieta Matrix and its Determinant}


\author[Ufuk]{Ufuk KAYA}

\ead[Ufuk]{mat-ufuk@hotmail.com}

\address[Ufuk]{Department of Mathematics, Faculty of Arts and Sciences, Bitlis Eren University, Bitlis, Turkey}

\begin{abstract}
In this paper, we define a matrix which we call Vieta matrix and calculate its determinant:
\[
\resizebox{1\hsize}{!}{$
\left(
\begin{array}{cccc}
1&1&\cdots&1\\
a_{2}+a_{3}+\cdots+a_{n}&a_{1}+a_{3}+\cdots+a_{n}&\cdots&a_{1}+a_{2}+\cdots+a_{n-1}\\
a_{2}a_{3}+a_{2}a_{4}+\cdots+a_{n-1}a_{n}&a_{1}a_{3}+a_{1}a_{4}+\cdots+a_{n-1}a_{n}&\cdots&a_{1}a_{2}+a_{1}a_{3}+\cdots+a_{n-2}a_{n-1}\\
\vdots&\vdots&\ddots&\vdots\\
a_{2}a_{3}\dots a_{n-1}a_{n}&a_{1}a_{3}\dots a_{n-1}a_{n}&\cdots&a_{1}a_{2}\dots a_{n-2}a_{n-1}
\end{array}
\right)$}
\]
\end{abstract}


\end{frontmatter}


\section{Introduction}
\label{int}

The determinant of a square matrix can be calculated by Leibnitz formula or Laplace formula. Especially, the determinant of a $3\times 3$ matrix can be calculated by rule of Sarrus. Generally, the determinants of $3\times 3$ matrices can be easily calculated, but $4\times 4$ and more are so difficult. In some cases, the determinants of $n\times n$ matrices can be formulated, such as Vandermonde matrix:
\[
\left(
\begin{array}{ccccc}
1&1&1&\cdots&1\\
a_{1}&a_{2}&a_{3}&\cdots&a_{n}\\
a_{1}^{2}&a_{2}^{2}&a_{3}^{2}&\cdots&a_{n}^{2}\\
\vdots&\vdots&\vdots&\ddots&\vdots\\
a_{1}^{n-1}&a_{2}^{n-1}&a_{3}^{n-1}&\cdots&a_{n}^{n-1}
\end{array}
\right).
\]
The determinant of Vandermonde matrix is
\[
\prod\limits_{k>i}\left(a_{k}-a_{i}\right).
\]
In this paper, we define a matrix which we call Vieta matrix:
\begin{equation}\label{1}
\resizebox{1\hsize}{!}{$
\left(
\begin{array}{ccccc}
1&1&1&\cdots&1\\
a_{2}+a_{3}+\cdots+a_{n}&a_{1}+a_{3}+\cdots+a_{n}&a_{1}+a_{2}+\cdots+a_{n}&\cdots&a_{1}+a_{2}+\cdots+a_{n-1}\\
a_{2}a_{3}+a_{2}a_{4}+\cdots+a_{n-1}a_{n}&a_{1}a_{3}+a_{1}a_{4}+\cdots+a_{n-1}a_{n}&a_{1}a_{2}+a_{1}a_{4}+\cdots+a_{n-1}a_{n}&\cdots&a_{1}a_{2}+a_{1}a_{3}+\cdots+a_{n-2}a_{n-1}\\
\vdots&\vdots&\vdots&\ddots&\vdots\\
a_{2}a_{3}\dots a_{n-1}a_{n}&a_{1}a_{3}\dots a_{n-1}a_{n}&a_{1}a_{2}\dots a_{n-1}a_{n}&\cdots&a_{1}a_{2}\dots a_{n-2}a_{n-1}
\end{array}
\right)$}
\end{equation}
and we calculate its determinant. Also, we give some applications of it.

\section{The Determinant of Vieta Matrix}
\label{determinant}

\begin{definition}
We state that the matrix given by the formula (\ref{1}) is the Vieta matrix of the complex numbers $a_{1},a_{2},\dots,a_{n}$. It is equal to the following
\[
\left(
\begin{array}{cc}
1&1\\
a_{2}&a_{1}
\end{array}
\right)
\]
for $n=2$,
\[
\left(
\begin{array}{ccc}
1&1&1\\
a_{2}+a_{3}&a_{1}+a_{3}&a_{1}+a_{2}\\
a_{2}a_{3}&a_{1}a_{3}&a_{1}a_{2}
\end{array}
\right)
\]
for $n=3$ and
\[
\resizebox{1\hsize}{!}{$
\left(
\begin{array}{cccc}
1&1&1&1\\
a_{2}+a_{3}+a_{4}&a_{1}+a_{3}+a_{4}&a_{1}+a_{2}+a_{4}&a_{1}+a_{2}+a_{3}\\
a_{2}a_{3}+a_{2}a_{4}+a_{3}a_{4}&a_{1}a_{3}+a_{1}a_{4}+a_{3}a_{4}&a_{1}a_{2}+a_{1}a_{4}+a_{2}a_{4}&a_{1}a_{2}+a_{1}a_{3}+a_{2}a_{3}\\
a_{2}a_{3}a_{4}&a_{1}a_{3}a_{4}&a_{1}a_{2}a_{4}&a_{1}a_{2}a_{3}
\end{array}
\right)$}
\]
for $n=4$.
\end{definition}
\noindent
One can easily calculate that the determinants of the above matrices are
\[
a_{1}-a_{2},
\]
\[
\left(a_{1}-a_{2}\right)\left(a_{1}-a_{3}\right)\left(a_{2}-a_{3}\right),
\]
and
\[
\left(a_{1}-a_{2}\right)\left(a_{1}-a_{3}\right)\left(a_{1}-a_{4}\right)\left(a_{2}-a_{3}\right)\left(a_{2}-a_{4}\right)\left(a_{3}-a_{4}\right)
\]
respectively. Now we prove that the determinant of the Vieta matrix is
\begin{equation}\label{2}
\prod\limits_{1\le i<k\le n}\left(a_{i}-a_{k}\right)
\end{equation}
in the general case by the following theorem:
\begin{theorem}\label{theorem1}
The determinant of the Vieta matrix (\ref{1}) is equal to the product in the formula (\ref{2}).
\end{theorem}
\begin{proof}
It is obvious that if at least two numbers of $a_{1},a_{2},\dots,a_{n}$ are equal to 0, then the last row of the matrix (\ref{1}) is the zero vector, and so, the determinant is 0. Assume that the numbers $a_{1},a_{2},\dots,a_{n}$ include at most one zero. We now use the induction method. For $n=2$, the determinant of the Vieta matrix
\[
\left(
\begin{array}{cc}
1&1\\
a_{2}&a_{1}
\end{array}
\right)
\]
is $a_{1}-a_{2}$. We assume that the assertion is valid for $n$, i.e. the determinant of the matrix (\ref{1}) is equal to the expression (\ref{2}) for any $n$ complex numbers. We show that the assertion is valid for $n+1$. First, we assume that $a_{n+1}=0$. Then, we have
\[
\resizebox{1\hsize}{!}{$
\left|
\begin{array}{ccccc}
1&1&\cdots&1&1\\
a_{2}+a_{3}+\cdots+a_{n}&a_{1}+a_{3}+\cdots+a_{n}&\cdots&a_{1}+a_{2}+\cdots+a_{n-1}&a_{1}+a_{2}+\cdots+a_{n}\\
a_{2}a_{3}+a_{2}a_{4}+\cdots+a_{n-1}a_{n}&a_{1}a_{3}+a_{1}a_{4}+\cdots+a_{n-1}a_{n}&\cdots&a_{1}a_{2}+a_{1}a_{3}+\cdots+a_{n-2}a_{n-1}&a_{1}a_{2}+a_{1}a_{3}+\cdots+a_{n-1}a_{n}\\
\vdots&\vdots&\ddots&\vdots&\vdots\\
0&0&\cdots&0&a_{1}a_{2}\dots a_{n-1}a_{n}
\end{array}
\right|$}
\]
\[
\resizebox{1\hsize}{!}{$
=a_{1}a_{2}\dots a_{n}\left|
\begin{array}{ccccc}
1&1&\cdots&1\\
a_{2}+a_{3}+\cdots+a_{n}&a_{1}+a_{3}+\cdots+a_{n}&\cdots&a_{1}+a_{2}+\cdots+a_{n-1}\\
a_{2}a_{3}+a_{2}a_{4}+\cdots+a_{n-1}a_{n}&a_{1}a_{3}+a_{1}a_{4}+\cdots+a_{n-1}a_{n}&\cdots&a_{1}a_{2}+a_{1}a_{3}+\cdots+a_{n-2}a_{n-1}\\
\vdots&\vdots&\ddots&\vdots\\
a_{2}a_{3}\dots a_{n-1}a_{n}&a_{1}a_{3}\dots a_{n-1}a_{n}&\cdots&a_{1}a_{2}\dots a_{n-2}a_{n-1}
\end{array}
\right|$}
\]
\[
=a_{1}a_{2}\dots a_{n}\prod\limits_{1\le i<k\le n}\left(a_{i}-a_{k}\right)=\prod\limits_{1\le i<k\le n+1}\left(a_{i}-a_{k}\right).
\]
Now, we assume that $a_{n+1}\ne0$ and define a function $f:\mathbb{C}\to\mathbb{C}$ by the following:
\[
f\left(x\right)=
\]
\begin{equation}\label{3}
\resizebox{1\hsize}{!}{$
\left|
\begin{array}{ccccc}
1&1&\cdots&1&1\\
a_{2}+a_{3}+\cdots+a_{n}+x&a_{1}+a_{3}+\cdots+a_{n}+x&\cdots&a_{1}+a_{2}+\cdots+a_{n-1}+x&a_{1}+a_{2}+\cdots+a_{n}\\
a_{2}a_{3}+a_{2}a_{4}+\cdots+a_{2}x+a_{3}a_{4}+\cdots+a_{n}x&a_{1}a_{3}+a_{1}a_{4}+\cdots+a_{1}x+a_{3}a_{4}+\cdots+a_{n}x&\cdots&a_{1}a_{2}+a_{1}a_{3}+\cdots+a_{1}x+a_{2}a_{3}+\cdots+a_{n-1}x&a_{1}a_{2}+a_{1}a_{3}+\cdots+a_{n-1}a_{n}\\
\vdots&\vdots&\ddots&\vdots&\vdots\\
a_{2}a_{3}\dots a_{n}x&a_{1}a_{3}\dots a_{n}x&\cdots&a_{1}a_{2}\dots a_{n-1}x&a_{1}a_{2}\dots a_{n-1}a_{n}
\end{array}
\right|$}
\end{equation}
Note that $f$ is the Vieta determinant of the numbers $a_{1},a_{2},\dots,a_{n},x$. It is obvious that the function $f$ is a polynomial whose degree is at most $n$ and it is equal to zero at the points $a_{1},a_{2},\dots a_{n}$. Then, it has the form
\begin{equation}\label{4}
f\left(x\right)=A\left(x-a_{1}\right)\left(x-a_{2}\right)\dots\left(x-a_{n}\right),
\end{equation}
where $A$ is a complex constant. We now find the number $A$. By writing $x=0$ in (\ref{3}) and (\ref{4}), we have the following
\[
A\left(-1\right)^{n}a_{1}a_{2}\dots a_{n}=f\left(0\right)=
\]
\[
\resizebox{1\hsize}{!}{$
\left|
\begin{array}{ccccc}
1&1&\cdots&1&1\\
a_{2}+a_{3}+\cdots+a_{n}&a_{1}+a_{3}+\cdots+a_{n}&\cdots&a_{1}+a_{2}+\cdots+a_{n-1}&a_{1}+a_{2}+\cdots+a_{n}\\
a_{2}a_{3}+a_{2}a_{4}+\cdots+a_{n-1}a_{n}&a_{1}a_{3}+a_{1}a_{4}+\cdots+a_{n-1}a_{n}&\cdots&a_{1}a_{2}+a_{1}a_{3}+\cdots+a_{n-2}a_{n-1}&a_{1}a_{2}+a_{1}a_{3}+\cdots+a_{n-1}a_{n}\\
\vdots&\vdots&\ddots&\vdots&\vdots\\
0&0&\cdots&0&a_{1}a_{2}\dots a_{n-1}a_{n}
\end{array}
\right|$}
\]
\[
\resizebox{1\hsize}{!}{$
=a_{1}a_{2}\dots a_{n}\left|
\begin{array}{ccccc}
1&1&\cdots&1\\
a_{2}+a_{3}+\cdots+a_{n}&a_{1}+a_{3}+\cdots+a_{n}&\cdots&a_{1}+a_{2}+\cdots+a_{n-1}\\
a_{2}a_{3}+a_{2}a_{4}+\cdots+a_{n-1}a_{n}&a_{1}a_{3}+a_{1}a_{4}+\cdots+a_{n-1}a_{n}&\cdots&a_{1}a_{2}+a_{1}a_{3}+\cdots+a_{n-2}a_{n-1}\\
\vdots&\vdots&\ddots&\vdots\\
a_{2}a_{3}\dots a_{n-1}a_{n}&a_{1}a_{3}\dots a_{n-1}a_{n}&\cdots&a_{1}a_{2}\dots a_{n-2}a_{n-1}
\end{array}
\right|$}
\]
\[
=a_{1}a_{2}\dots a_{n}\prod\limits_{1\le i<k\le n}\left(a_{i}-a_{k}\right).
\]
By the assumption $a_{1}a_{2}\dots a_{n}\ne0$, we obtain
\[
A=\left(-1\right)^{n}\prod\limits_{1\le i<k\le n}\left(a_{i}-a_{k}\right).
\]
Hence, we can rewrite the function $f$ as follows:
\[
f\left(x\right)=\left(\left(-1\right)^{n}\prod\limits_{1\le i<k\le n}\left(a_{i}-a_{k}\right)\right)\left(x-a_{1}\right)\left(x-a_{2}\right)\dots\left(x-a_{n}\right).
\]
If we write $x=a_{n+1}$ in the above expression, we have
\[
f\left(a_{n+1}\right)=\left(a_{n+1}-a_{1}\right)\left(a_{n+1}-a_{2}\right)\dots\left(a_{n+1}-a_{n}\right)\left(-1\right)^{n}\prod\limits_{1\le i<k\le n}\left(a_{i}-a_{k}\right)
\]
\[
=\prod\limits_{1\le i<k\le n+1}\left(a_{i}-a_{k}\right).
\]
This completes the proof.
\end{proof}

\begin{corollary}\label{corollary1}
Let $a_{1},a_{2},\dots,a_{n}$ and $c$ be arbitrary elements in an integral domain. Then the Vieta determinants of the elements $a_{1},a_{2},\dots,a_{n}$ and $a_{1}-c,a_{2}-c,\dots,a_{n}-c$ are equal to each other.
\end{corollary}

\section{Some Applications of Vieta Matrix}
\label{applications}

\subsection{Wronskian of some polynomials}

Let $a_{1},a_{2},\dots,a_{n}\in\mathbb{C}$. Consider the polynomials
\[
\begin{array}{l}
p_{1}\left(x\right)=\left(x-a_{2}\right)\left(x-a_{3}\right)\dots\left(x-a_{n}\right),\\
p_{2}\left(x\right)=\left(x-a_{1}\right)\left(x-a_{3}\right)\dots\left(x-a_{n}\right),\\
\vdots\\
p_{n}\left(x\right)=\left(x-a_{1}\right)\left(x-a_{2}\right)\dots\left(x-a_{n-1}\right).
\end{array}
\]
We now find the Wronskian of these polynomials, i.e. calculate the following
\begin{equation}\label{5}
W\left[p_{1},p_{2},\dots,p_{n}\right]\left(x\right)=
\left|
\begin{array}{cccc}
p_{1}\left(x\right)&p_{2}\left(x\right)&\cdots&p_{n}\left(x\right)\\
p_{1}^{\prime}\left(x\right)&p_{2}^{\prime}\left(x\right)&\cdots&p_{n}^{\prime}\left(x\right)\\
\vdots&\vdots&\ddots&\vdots\\
p_{1}^{\left(n-1\right)}\left(x\right)&p_{2}^{\left(n-1\right)}\left(x\right)&\cdots&p_{n}^{\left(n-1\right)}\left(x\right)
\end{array}
\right|.
\end{equation}
Note that the polynomials $p_{1},p_{2},\dots,p_{n}$ can be rewritten as follows:
\begin{equation}\label{6}
\resizebox{1\hsize}{!}{$
\begin{array}{l}
p_{1}\left(x\right)=x^{n-1}-\left(a_{2}+a_{3}+\cdots+a_{n}\right)x^{n-2}+\left(a_{2}a_{3}+a_{2}a_{4}+\cdots+a_{n-1}a_{n}\right)x^{n-3}+\cdots+\left(-1\right)^{n-1}a_{2}a_{3}\dots a_{n},\\
p_{2}\left(x\right)=x^{n-1}-\left(a_{1}+a_{3}+\cdots+a_{n}\right)x^{n-2}+\left(a_{1}a_{3}+a_{1}a_{4}+\cdots+a_{n-1}a_{n}\right)x^{n-3}+\cdots+\left(-1\right)^{n-1}a_{1}a_{3}\dots a_{n},\\
\vdots\\
p_{n}\left(x\right)=x^{n-1}-\left(a_{1}+a_{2}+\cdots+a_{n-1}\right)x^{n-2}+\left(a_{1}a_{2}+a_{1}a_{3}+\cdots+a_{n-2}a_{n-1}\right)x^{n-3}+\cdots+\left(-1\right)^{n-1}a_{1}a_{2}\dots a_{n-1}.
\end{array}
$}
\end{equation}
Furthermore, each of the functions $p_{1},p_{2},\dots,p_{n}$ is a solution of the linear differential equation $y^{n}=0$. Then, the Wronskian $W\left[p_{1},p_{2},\dots,p_{n}\right]\left(x\right)$ is independent of the variable $x$ by Abel's identity, see \cite{boyce}. So, it is equal to the value $W\left[p_{1},p_{2},\dots,p_{n}\right]\left(0\right)$, i.e.
\begin{equation}\label{7}
W\left[p_{1},p_{2},\dots,p_{n}\right]\left(x\right)=W\left[p_{1},p_{2},\dots,p_{n}\right]\left(0\right).
\end{equation}
By the relations (\ref{5}), (\ref{6}), (\ref{7}) and, Theorem \ref{theorem1} we have
\[
W\left[p_{1},p_{2},\dots,p_{n}\right]\left(x\right)=
\]
\[
\resizebox{1\hsize}{!}{$
\left|
\begin{array}{cccc}
\left(-1\right)^{n-1}a_{2}a_{3}\dots a_{n-1}a_{n}&\left(-1\right)^{n-1}a_{1}a_{3}\dots a_{n-1}a_{n}&\cdots&\left(-1\right)^{n-1}a_{1}a_{2}\dots a_{n-2}a_{n-1}\\
\vdots&\vdots&\ddots&\vdots\\
-\left(n-2\right)!\left(a_{2}+a_{3}+\cdots+a_{n}\right)&-\left(n-2\right)!\left(a_{1}+a_{3}+\cdots+a_{n}\right)&\cdots&-\left(n-2\right)!\left(a_{1}+a_{2}+\cdots+a_{n-1}\right)\\
\left(n-1\right)!&\left(n-1\right)!&\cdots&\left(n-1\right)!
\end{array}
\right|$}
\]
\[
\resizebox{1\hsize}{!}{$
=\prod\limits_{k=0}^{n-1}k!\left|
\begin{array}{cccc}
1&1&\cdots&1\\
a_{2}+a_{3}+\cdots+a_{n}&a_{1}+a_{3}+\cdots+a_{n}&\cdots&a_{1}+a_{2}+\cdots+a_{n-1}\\
\vdots&\vdots&\ddots&\vdots\\
a_{2}a_{3}\dots a_{n-1}a_{n}&a_{1}a_{3}\dots a_{n-1}a_{n}&\cdots&a_{1}a_{2}\dots a_{n-2}a_{n-1}
\end{array}
\right|$}
\]
\[
=\prod\limits_{k=0}^{n-1}k!\cdot\prod\limits_{1\le i<k\le n}\left(a_{i}-a_{k}\right).
\]

\subsection{Jacobian of some functions}

Consider the functions
\[
\begin{array}{l}
f_{1}\left(x_{1},x_{2},\dots,x_{n}\right)=x_{1}+x_{2}+\cdots+x_{n},\\
f_{2}\left(x_{1},x_{2},\dots,x_{n}\right)=x_{1}x_{2}+x_{1}x_{3}+\cdots+x_{n-1}x_{n},\\
\vdots\\
f_{n}\left(x_{1},x_{2},\dots,x_{n}\right)=x_{1}x_{2}\dots x_{n}.
\end{array}
\]
If we calculate the Jacobian of the functions $f_{1},f_{2},\dots,f_{n}$ by using Theorem \ref{theorem1}, we have
\[
\left|
\begin{array}{cccc}
\displaystyle{\frac{\partial f_{1}}{\partial x_{1}}}&\displaystyle{\frac{\partial f_{1}}{\partial x_{2}}}&\cdots&\displaystyle{\frac{\partial f_{1}}{\partial x_{n}}}\\
\displaystyle{\frac{\partial f_{2}}{\partial x_{1}}}&\displaystyle{\frac{\partial f_{2}}{\partial x_{2}}}&\cdots&\displaystyle{\frac{\partial f_{2}}{\partial x_{n}}}\\
\vdots&\vdots&\ddots&\vdots\\
\displaystyle{\frac{\partial f_{n}}{\partial x_{1}}}&\displaystyle{\frac{\partial f_{n}}{\partial x_{2}}}&\cdots&\displaystyle{\frac{\partial f_{n}}{\partial x_{n}}}
\end{array}
\right|
\]
\[
=\left|
\begin{array}{cccc}
1&1&\cdots&1\\
x_{2}+x_{3}+\cdots+x_{n}&x_{1}+x_{3}+\cdots+x_{n}&\cdots&x_{1}+x_{2}+\cdots+x_{n-1}\\
\vdots&\vdots&\ddots&\vdots\\
x_{2}x_{3}\dots x_{n-1}x_{n}&x_{1}x_{3}\dots x_{n-1}x_{n}&\cdots&x_{1}x_{2}\dots x_{n-2}x_{n-1}
\end{array}
\right|
\]
\[
=\prod\limits_{1\le i<k\le n}\left(x_{i}-x_{k}\right).
\]








\end{document}